\documentclass[12pt,oneside,english,reqno]{amsart}
\usepackage[T1]{fontenc}
\usepackage[latin9]{inputenc}
\usepackage{amsthm}
\usepackage{amstext}
\usepackage{amssymb}
\usepackage{esint}

\makeatletter
\numberwithin{equation}{section}
\numberwithin{figure}{section}
\theoremstyle{plain}
\newtheorem{thm}{\protect\theoremname}[section]
  \theoremstyle{definition}
  \newtheorem{problem}[thm]{\protect\problemname}
  \theoremstyle{plain}
  \newtheorem{prop}[thm]{\protect\propositionname}
  \theoremstyle{plain}
  \newtheorem{lem}[thm]{\protect\lemmaname}
  \theoremstyle{remark}
  \newtheorem{rem}[thm]{\protect\remarkname}

\makeatother

\usepackage{babel}
  \providecommand{\lemmaname}{Lemma}
  \providecommand{\problemname}{Problem}
  \providecommand{\propositionname}{Proposition}
  \providecommand{\remarkname}{Remark}
\providecommand{\theoremname}{Theorem}

\begin{document}

\title{On curvature pinching of conic 2-spheres}

\author{Hao Fang}

\address{14 MacLean Hall, University of Iowa, Iowa City, IA, 52242}

\email{hao-fang@uiowa.edu}

\thanks{H.F.'s work is partially supported by NSF DMS-100829. }

\author{Mijia Lai}

\address{800 Dongchuan RD, Shanghai Jiao Tong University, Shanghai, China,
200240 }

\email{laimijia@sjtu.edu.cn}

\thanks{M.L.'s work is partially supported by Shanghai Sailing Program No. 15YF1406200 and NSFC No. 11501360.}

\date{June 18, 2015}
\begin{abstract}
We study metrics on conic 2-spheres when no Einstein metrics exist.
In particular, when the curvature of a conic metric is positive, we
obtain the best curvature pinching constant. We also show that when
this best pinching constant is approached, the conic 2-sphere has
an explicit Gromov-Hausdorff limit. This is a generalization of the
previous results of Chen-Lin and Bartolucci for 2-spheres with one
or two conic points.
\end{abstract}

\maketitle

\section{Introduction}


Due to the recent development in Kähler geometry, there is a lot
of interest in the study of metrics with conic singularities along
a divisor \cite{D,JMR}. The deep connection between algebraic stability
and existence of Kähler-Einstein metrics on Fano manifolds is highlighted
in the recent solution of the Yau-Tian-Donaldson conjecture \cite{CDS1,CDS2,CDS3,Ti}.
Metrics with conic singularities play an essential role in this direction,
namely it is more natural to consider Kähler-Einstein metrics with
possible conic singularities on Fano manifolds. It is thus interesting
to explore further geometric properties of conic manifolds. On the
other hand, manifolds satisfying certain unstability conditions do
not admit Kähler-Einstein metrics or constant scalar curvature metrics.
Therefore, an interesting problem is to search for other ``canonical''
Kähler metrics (possibly with conical singularities) on such manifolds. In this
article, we search for ``least-pinched'' metrics on conic 2-spheres when no Einstein
metrics exist.

For surfaces with conic metric singularities, or conic surfaces, the
relation between stability and existence of Einstein metrics can be
expressed in explicit forms. In order to study the prescribing curvature
problem for conic surfaces, Troyanov \cite{Tr} has classified conic
surfaces into sub-critical, critical and super-critical categories.
On the other hand, the logarithmical K-stability \cite{RT} is linked
to the coercivity of twisted Mabuchi $K$-energy functional, which
means any conic surface is either logarithmically K-stable, semi-stable
or unstable. It is shown that these two classifications coincide \cite{RT}.

Let us start with some definitions before further elaboration. For
a closed Riemann surface $S$, a metric $g$ is said to
have a conic singularity of order $\beta\in(-1,\infty)$ at a point
$p\in S$ if under a local holomorphic coordinate $\{z\}$ centered
at $p$,
\[
g=e^{f(z)}|z|^{2\beta}|dz|^{2},
\]
where $f(z)$ is locally bounded and $C^{2}$ away from $p$.
The conic singularity is modeled on the Euclidean cone:
${\mathbb{C}}$ with a metric $|z|^{2\beta}|dz|^{2}$ is isometric
to a Euclidean cone of angle $2\pi(\beta+1)$ with the cone tip at
the origin.

In general, we shall use the triple $(S, D, g)$ to denote a
closed orientable Riemannian surface $S$ with a conic metric
$g$, for which the information of its singularities is encoded in the divisor $D=\sum_{i=1}^{n}\beta_{i}p_{i}$
in an obvious manner, i.e., $g$ has conic singularities at $p_i$ of order $\beta_i$.
We sometimes say that the conic metric $g$ represents $D$.

For such
a triple $(S,D,g)$, let $K=K(g)$ be the Gaussian curvature
of $g$ defined on $S\backslash D$. Throughout the paper,
we assume $K$ can be extended to a H\"older continuous function
on $S$. The collection of all such conic metrics representing
$D$ is denoted by $\mathbb{\mathcal{C}}(S,D)$. We shall use the
pair $(S, D)$ to denote a conic surface when the metric is not
specified.

The Gauss-Bonnet formula for the conic surface $(S,D)$ becomes
\cite{Tr}
\[
\int_{S}K(g)dA(g)=2\pi\chi(S,D):=2\pi(\chi+|D|),
\]
where $|D|=\sum_{i=1}^{n} \beta_i$ is the degree of the divisor.

In his seminal paper \cite{Tr}, Troyanov systematically studies the
prescribing curvature problem on the conic surface $(S,D)$. He divided
the problem into three cases according to the sign of the Euler characteristic number
$\chi(S,D).$ For $\chi(S,D)>0$, he further classified the problem
into the following three cases:
\begin{enumerate}
\item subcritical case: $\chi(S,D)<\min\{2,2+2\min_{i}\beta_{i}\}$;
\item critical case: $\chi(S,D)=\min\{2,2+2\min_{i}\beta_{i}\};$
\item supercritical case: $\chi(S,D)>\min\{2,2+2\min_{i}\beta_{i}\}.$
\end{enumerate}
The constant $\min\{2,2+2\min_{i}\beta_{i}\}$ turns out to be the Trudinger
constant \cite{Tr} in the corresponding
Moser-Trudinger inequality for conic surfaces, which plays an important
role in the prescribing curvature problem.
According to Ross-Thomas \cite{RT}, a conic surface $(S,D)$ being
subcritical, critical or supercritical, respectively can be reinterpreted
as it being logarithmically K-stable, semi-stable or unstable, respectively.
For further developments of the prescribing curvature problem, we refer
readers to \cite{BDM,BaMa,E} and the references therein.

Now let us examine the Yamabe problem on conic surfaces, namely the existence
of constant curvature metrics on $(S,D)$. If $\chi(S,D)\leq0$, there always exists a conic metric representing $D$
with constant curvature \cite{Tr}. When $\chi(S,D)>0$ and all $\beta_{i}\in(-1,0)$ (which we assume from now on),
$S$ is necessarily a 2-sphere. In this situation, $S$ admits a conic
metric of positive constant curvature if and only if:
\begin{enumerate}
\item either $n=2$ and $D=\beta_1p+\beta_2q$ with $\beta_{1}=\beta_{2}$; \label{football}
\item or $n\geq3$ and $D$ is subcritical i.e., $\chi(S,\beta)<\min\{2,2+\min_{i}\beta_{i}$\}.\label{subcritical}
\end{enumerate}
Note that surfaces satisfying  (\ref{football}) are critical and
are often called (American) footballs, see \cite{CL2} for some classification
results. For surfaces satisfying (\ref{subcritical}), the sufficiency
is proved by Troyanov \cite{Tr}, the necessity and uniqueness argument
is due to Luo-Tian \cite{LT}.



In view of above results, there are two cases of conic 2-spheres which do not carry
metrics of constant (positive) curvature:
\begin{enumerate}
  \item $D$ is supercritical;
  \item $D$ is critical and $n\geq 3$.
\end{enumerate}

It is then natural to seek for other ``canonical'' metrics as substitutes for constant
curvature metrics. Since in two-dimension the curvature is a scalar function, we can consider
from the viewpoint of ``least-pinched'' metrics. More precisely, we ask the following question:
\begin{problem}
\label{prob:problem} For a conic 2-sphere $(S, D, g)$
with positive Gaussian curvature $K(g)$ (assumed to be extended to a continuous function on $S$), let $K_{\max}$ and $K_{\min}$
denote the maximum and the minimum of $K(g)$, respectively. Define
the curvature pinching constant of $g$ as
\begin{equation}
\rho(g)=\frac{K_{\min}}{K_{\max}}.\label{eq:ratio}
\end{equation}
What is $\text{\ensuremath{{\displaystyle \sup_{g\in\mathcal{C}(S,D)}}}}\{\rho(g)\}$?\label{prob1}
\end{problem}

If $(S, D)$ admits a constant curvature metric $g$, then clearly $\sup \{\rho(g)\}=1$, which
is attained by constant curvature metrics. So nontrivial cases for this problem are
\begin{enumerate}
  \item $D$ is supercritical;
  \item $D$ is critical and $n\geq 3$.
\end{enumerate}

Problem \ref{prob1} was first asked by Thurston \cite{Th} in 1978
for 2-spheres with one or two conic points (both are supercritical). Bartolucci \cite{B} answers
Thurston's question based on the analysis of Chen-Lin \cite{ChLi},
who have treated the one conic point case. More precisely, Bartolucci
has proved the following
\begin{thm}[Bartolucci]
\label{thm:old}Let $(S^2, D)$ be a conic 2-sphere, with $D=\alpha p+\beta q$.
Suppose $-1<\beta<\alpha\leq0$ ($\alpha=0$ corresponds to the case
of one conic point), then for all piecewise smooth
and $C^{1,1}$ conic metrics $g$ on the $S^{2}$ representing $D$,
\begin{equation}
\rho(g)\leq\frac{(\beta+1)^{2}}{(\alpha+1)^{2}},\label{eq:add1}
\end{equation}
where the equality holds if and only if $(S^2, D,g)$( up to a M\"obius transformation ) is the ``glued football''.
\end{thm}

For $-1<\beta\leq\alpha\leq0,$ a glued football $S_{\alpha,\beta}^{2}$
is given by the following conformal metric $g=e^{2u}g_{0}$ on the
2-sphere, where $g_{0}$ is the standard Euclidean metric on $\mathbb{C}$
and the conformal factor $u$ is defined as

\begin{equation}
u=u_{\alpha,\beta}=\begin{cases}
\ln(\frac{2(\alpha+1)r^{\text{\ensuremath{\alpha}}}}{1+r^{2+2\alpha}}), & r<1;\\
\\
\ln(\frac{2(\alpha+1)r^{\beta}}{1+r^{2+2\beta}}) & r\geq1,
\end{cases}\label{eq:italian football}
\end{equation}
with $r=|z|.$

The glued football $S_{\alpha,\beta}^{2}$ has a conic singularity
along $\beta\infty$ when $\alpha=0$, or
$\alpha0+\beta\infty$, when $\alpha<0$. 

We note that if $\alpha=\beta$ this is a smooth conic metric away
from $z=0,\infty$; it has constant curvature 1. This is the so-called
(American) football.

When $\alpha\neq\beta,$ $S_{\alpha,\beta}^{2}$ is glued by two footballs
of different angles along their equator, and thus has piecewise constant
curvature:
\begin{equation}
K(g)=\begin{cases}
1, & r<1;\\
\\
\frac{(\beta+1)^{2}}{(\alpha+1)^{2}}, & r>1.
\end{cases}\label{eq:italiancurvature}
\end{equation}

It follows $\rho(g)=\frac{(\beta+1)^{2}}{(\alpha+1)^{2}}$, which realizes the equality in (\ref{eq:add1}).


In this paper, we answer Problem \ref{prob:problem} in its full generality.
Our main result is the following
\begin{thm}
\label{mainthm}Let $(S,D)$ be a conic 2-sphere where $D=\sum_{i=1}^{n}\beta_i p_i$ is supercritical,
assume $\beta_{1}=\min_i{\beta_{i}}$ and let $\alpha=|D|-\beta_{1}$. Then $\forall g\in\mathcal{C}(S, D)$,
\begin{equation}
\rho(g)<\frac{(1+\beta_{1})^{2}}{(1+\alpha)^{2}}:=\rho_{0}(S,D), \label{eq:bestpinching}
\end{equation}
moreover $\rho_{0}$ is optimal.
\end{thm}
Note if $D$ is supercritical, then $\beta_{1}<\alpha$, consequently $\rho_{0}(S, D)<1$.
Hence this gives a qualitative evidence for the non-existence of constant
positive curvature on supercritical conic 2-spheres, or it can be
viewed as a necessary condition for the Nirenberg problem in the supercritical
conic 2-sphere setting. The theorem also recovers Bartolucci's result if $D$ has one
or two conic points.

We find that glued footballs also serve as the extremal geometrically, when $n\geq 3$.
\begin{thm}
\label{thm:converging}Let $(S,D)$ be a conic 2-sphere where $D=\sum_{i=1}^{n}\beta_i p_i$ is supercritical,
assume $\beta_{1}=\min_i{\beta_{i}}$ and let $\alpha=|D|-\beta_{1}$.
Then for any sequence of smooth conic metrics $\{g_{i}\}_{i=1}^{\infty}$
with $\lim_{i\to\infty}\rho(g_{i})=\rho_{0}$ and $\max K(g_{i})=1$,
$\{(S,D,g_{i})\}$ converges to the glued football
$(S_{\alpha,\beta_{1}}^{2}, \hat{D}=\beta_1 p+\alpha q)$ in the Gromov-Hausdorff sense, moreover
$p_1\to p$ and $p_2, \cdots p_n\to q$ along the convergence.
\end{thm}

For the critical case, following identical arguments, we have
\begin{thm}
\label{thm:critical}Let $(S, D)$ be a conic 2-sphere where $D=\sum_{i=1}^{n}\beta_i p_i$ is critical and $n\geq3$, assume $\beta_{1}=\min{\beta_{i}}$ and let $\alpha=|D|-\beta_{1}$. Then $\forall g\in\mathcal{C}(S, D)$,
\[
\rho(g)<\frac{(1+\beta_{1})^{2}}{(1+\alpha)^{2}}:=\rho_{0}(S,D)=1,
\]
and the constant $\rho_{0}(S,D)$ is optimal; moreover, for any
sequence of smooth conic metrics $g_i$ with $\lim_{i\to\infty}\rho(g_{i})=1$
and $\max K_{i}=1$,  $\{(S, D, g_i)\}$ converges to the football $(S_{\alpha,\alpha}^{2}, \hat{D}=\alpha p+\alpha q)$ in the Gromov-Hausdorff sense, and $p_1\to p$ and
$p_2, \cdots p_n\to q$ along the convergence.
\end{thm}
Our main tool of the proof is to apply the isoperimetric inequality
to obtain sharp differential inequalities. This is inspired by methods
first used by Chen-Lin \cite{ChLi} that are later extended by Bartolucci \cite{B}.
However, since there can be no simple symmetric rearrangement procedure for
the multi-singularity case, their method can not be applied directly.
Instead, we apply a similar argument used in our earlier paper \cite{FL}
to simplify their proof. In particular, we are able to derive a differential
inequality involving only the level sets of the conformal factor
without the symmetric rearrangement argument. 
On the other hand, the characterization of the equality case is obtained in \cite{ChLi,B}
by the more involved rearrangement technique. We need to do a finer analysis of the defect of isoperimetric inequalities as in
\cite{FL} to prove the convergence result in Theorem\ref{thm:converging},
which indicates all but one conic points merge in the limit procedure.

While previous works of \cite{ChLi,B} focus on $C^{1,1}$ and piecewise
smooth metrics, for which glued footballs serve as the unique geometric
sharp examples up to M\"obius transformations, it is clear that (\ref{eq:bestpinching})
is a strict inequality for smooth conic metrics. Using our previous results
in \cite{FL} and some careful analysis, we are able to construct
examples in general cases that (\ref{eq:bestpinching}) is indeed
sharp.

In recent preprints \cite{PSSW1,PSSW2}, the Ricci flow on conic 2-spheres
is shown to converge in stable, semi-stable and unstable cases. It
is shown in \cite{PSSW1} that conic metrics on a sphere with $n$
($n\geq3$) conic points in the semi-stable case converge in Gromov-Hausdorff
topology to a football along the Ricci flow. Also in \cite{PSSW2}
it is proved that the Ricci flow of an unstable conic metric converges
geometrically to a Ricci soliton  with two conic points. 
Thus, in the sense of the Ricci flow, Ricci solitons are considered to be canonical metrics
for conic 2-spheres.

Comparing the convergence results of Theorem~\ref{thm:converging} and Theorem~\ref{thm:critical} to convergence results of the Ricci flow to Ricci solitons,
we observe an interesting common feature: the conic point with the smallest cone angle remains in the limit process, while
all other conic points merge to form a new conic point.  Their difference is however more obvious: while Ricci solitons are
smooth with varied curvature, the glued footballs have piece-wise constant curvature. This is somewhat expected. As the problem of searching for metrics realizing $\sup_{g\in\mathcal{C}(S,D)} \{\rho(g)\}$ is not variational, and thus the extremal usually loses smoothness.


In general, if a Kähler manifold does not carry
metrics of constant scalar curvature in a fixed Kähler class, it would be very interesting
to see if the best scalar curvature pinching can be computed. Hopefully
we can address this problem in future works.

From a more analytic point of view, as mentioned earlier, Theorem \ref{mainthm}
gives a necessary condition for the Nirenberg problem, which would
be an interesting topic for further discussion. It also indicates
a sharp Moser-Trudinger inequality for supercritical conic 2-spheres.

This paper is organized as follows: in Section 2, we present our key
estimate for proving Theorem \ref{mainthm}; in Section 3 and Section
4, we prove Theorems \ref{mainthm} and \ref{thm:converging}, respectively.

Acknowledgements: Both authors would like to thank Bartolucci for
bringing up his work \cite{B} to their attention. They also thank the referee for pointing out some technical inaccuracy in a previous draft.

\section{Key Estimate}

In this section we give a proper set-up for Theorem \ref{mainthm}
and present a key estimate for its proof.

We start with some definitions.

Let $g_{0}$ be the standard Euclidean metric on $\mathbb{C}$. We
identify any point $p\in S^{2}$ with $z\in\mathbb{C}^{*}$ via stereographic
projection. For any given natural number $n$, consider a divisor
$D={\displaystyle \sum_{1}^{n}}\beta_{i}p_{i}$ where $p_{i}\in S^{2}$
and $-1<\beta_{i}<0,$ $i=1,\cdots,n.$ For simplicity, we assume
\begin{align}
\beta_{1} & \leq\beta_{2}\leq\cdots\leq\beta_{n}.\label{eq:add2}
\end{align}
Remember that
\begin{equation}
|D|=\sum_{1}^{n}\beta_{i}.\label{eq:add3}
\end{equation}
 We let
\begin{equation}
\alpha:=|D|-\beta_{1}=\sum_{2}^{n}\beta_{i}.\label{eq:add4}
\end{equation}
The supercritical condition implies that $\alpha>\beta_{1}$. Also,
without loss of generality, we may assume $z_{1}=\infty$.

Thus, a conformal metric $g=e^{2u}g_{0}$ represents $(S^{2},D)$
if the asymptotic behavior of $u$ near $z_{i}$ is:
\begin{itemize}
\item $u\sim\beta_{i}\ln|z-z_{i}|$ as $z\to z_{i}$, $i>1$;
\item $u\sim-(\beta_{1}+2)\ln|z|$ as $|z|\to z_{1}=\infty$.
\end{itemize}
The Gaussian curvature of $g$ is computed as
\begin{equation}
K(g)=-e^{-2u}\Delta u,\label{eq:gauss}
\end{equation}
when $z\neq z_{i}.$ Here the Laplacian is with respect to $g_{0}.$

The main result of this section is the following estimate:
\begin{prop}
\label{prop:keyestimate}Notations as above. Let $K$ be a positive continuous function on $\mathbb{C}$ satisfying (\ref{eq:gauss}) such that
\begin{equation}
0<a\leq K\leq b\label{eq:bound},
\end{equation}
then
\begin{equation}
\frac{a}{b}\leq\rho_{0}(S,D)=\frac{(1+\beta_{1})^{2}}{(1+\alpha)^{2}}.\label{eq:upperbound}
\end{equation}
\end{prop}
\begin{proof}
We consider the level sets of $u$. Using notations of our earlier
works \cite{FL}, we define,
\begin{align*}
\Omega(t) & :=\{u\geq t\}\subset\mathbb{C},\ \ \ \ \ A(t):=\int_{\Omega_{t}}Ke^{2u},\ \ \ \ \ B(t):=|\Omega_{t}|,
\end{align*}
where integrals are with respect to the Euclidean metric $g_{0}$ and $|\cdot|$ stands for the Lebesgue measure.
In view of the asymptotic behavior of $u$ near $\infty$, $B(t)$ is finite for any $t\in\mathbb{R}$.  The Gauss-Bonnet formula now reads as
\begin{equation}
\int_{\mathbb{R}^{2}}Ke^{2u}=2\pi(2+|D|)=\lim_{t\to-\infty}A(t).\label{eq:limitGB}
\end{equation}

Moreover, the asymptotic behavior of $u$ at singularities $z_2, \cdots z_n$ implies that $z_i\in\Omega(t)$, for $2\leq i\leq n$ and for all $t\in\mathbb{R}$.
Thus using equation (\ref{eq:gauss}), we have
\[
A(t)=\int_{\Omega(t)}Ke^{2u}=\int_{\Omega(t)}-\Delta u=\int_{\partial\Omega(t)}|\nabla u|+2\pi\alpha.
\]

It is easy to from definition that both $A(t)$ and $B(t)$ are non-increasing. It follows that $A'(t)$ and $B'(t)$ exist for almost everywhere $t$.

For such a $t$,
\begin{align}\notag
A'(t)&=\lim_{s\to t^{-}} \frac{\int_{\Omega(s)\setminus \Omega(t)} Ke^{2u}}{s-t}\geq\lim_{s\to t^{-}}\frac{be^{2u(z^{*})}|\Omega(s)\setminus \Omega(t)|}{s-t}=be^{2t} B'(t),
\end{align}
where we have applied mean value theorem with some $z^{*}\in \Omega(s)\setminus \Omega(s)$ and (\ref{eq:bound}). Similarly, $A'(t)\leq ae^{2t} B'(t)$ for almost everywhere $t$, so
\begin{equation}
\frac{1}{b}\leq\frac{e^{2t}B'(t)}{A'(t)}\leq\frac{1}{a}.\label{eq:1}
\end{equation}

It follows from the co-area formula (see Lemma 2.3 in ~\cite{BZ}) that for given $t_1<t_2$,
\begin{align} \label{eq:4}
B(t_1)-B(t_2)=|C\cap u^{-1}((t_1,t_2))|+\int_{t_1}^{t_2} \int_{\{u=\tau\}} \frac{1}{|\nabla u|} d\mathcal{H}^1 d\tau,
\end{align}
where $C$ denotes the set of critical points of $u$, i.e., $C=\{ z|\nabla u(z)=0\}$.

We claim $|C\cap u^{-1}((t_1,t_2))|=0$, for any $t_1<t_2$, which indicates that $B$ is absolutely continuous. By (\ref{eq:gauss}) and (\ref{eq:bound}), then for any $z_0=(x_0,y_0)\in \{C\cap u^{-1}((t_1,t_2))\}$, we may assume, without loss of generality, that $u_{xx}(z_0)\neq 0$. By virtue of the implicit function theorem, there exists some $\rho>0$ and $g: (y_0-\rho,y_0+\rho)\to \mathbb{R}$ such that $\{u_x(x,y)=0\cap B_{\rho}(z_0)\}$ is the graph of the function $x=g(y)$. Clearly,
\[
C\cap B_{\rho}(z_0)\subset \{u_x(x,y)=0\}\cap B_{\rho}(z_0), 
\]
which has 0 measure.  We have thus established the claim. $B(t)$ is absolutely continuous on any finite interval $[t_1,t_2]$.

From now on we may assume that computation is done for a generic $t$, i.e., for which $A'(t)$ and $B'(t)$ exist and (\ref{eq:1}) and (\ref{eq:4}) hold.

By the isoperimetric inequality and the H\"{o}lder's inequality,
we have
\begin{align}
4\pi B(t)\leq(\int_{\partial\Omega(t)}1)^{2}\leq\int_{\partial\Omega}|\nabla u|\int_{\partial\Omega(t)}\frac{1}{|\nabla u|}\leq-B'(t)(A(t)-2\pi\alpha),\label{2}
\end{align}
which leads to
\begin{equation}
\frac{B(t)}{B'(t)}\geq-\frac{A(t)-2\pi\alpha}{4\pi}.\label{eq:2.5}
\end{equation}

Combining (\ref{eq:1}) and (\ref{eq:2.5}), we get
\begin{align}
\notag\frac{d}{dt}[e^{2t}B(t)] & =e^{2t}B'(t)+2e^{2t}B(t)=e^{2t}B'(t)(1+\frac{2B(t)}{B'(t)})\nonumber \\
\notag & \leq\frac{e^{2t}B'(t)}{A'(t)}A'(t)(1-\frac{A-2\pi\alpha}{2\pi})\\
\leq & \begin{cases}
\frac{A'}{b}(1+\alpha-\frac{A}{2\pi}), & A\leq2\pi(1+\alpha);\\
\\
\frac{A'}{a}(1+\alpha-\frac{A}{2\pi}), & A>2\pi(1+\alpha).
\end{cases}\label{eq:add30}
\end{align}

Since $A(-\infty)=2\pi(2+|D|)>2\pi(1+\alpha)$, there exists $t_{0}\in\mathbb{R}$ such that $A(t_{0})=2\pi(1+\alpha)$. For
$t<t_{0}<T$, integrating (\ref{eq:add30}) from $t$ to $T$, and noticing $e^{2t}B(t)$ is absolutely continuous as well, we get
\begin{equation}
e^{2T}B(T)-e^{2t}B(t)\leq\frac{A^{2}(t)}{4a\pi}-\frac{1+\alpha}{a}A(t)+(1+\alpha)^{2}(\frac{\pi}{a}-\frac{\pi}{b}).\label{eq:3}
\end{equation}

Since
\[
\int_{\mathbb{R}^2} e^{2u}=2 \int_{-\infty}^{\infty} B(t)e^{2t} dt <\infty,
\] there exist sequences $t_n \to -\infty$ and $T_n\to \infty$, such that $e^{2t_n}B(t_n)\to 0$ and $e^{2T_n}B(T_n)\to 0$, respectively.

Plugging such two sequences into (\ref{eq:3}), we infer
\[
\frac{2+|D|}{1+\alpha}\geq1+\sqrt{\frac{a}{b}},
\]
which implies (\ref{eq:upperbound}). We have thus finished the proof.
\end{proof}

\section{Proof of main result}

In this section, we prove that the constant $\rho_{0}(S,D)$ obtained
in Proposition \ref{prop:keyestimate} is optimal.


The main result of this section is the following:
\begin{thm}
\label{thm:construction}For any fixed supercritical conic 2-sphere
$(S,D)$ and any $\epsilon>0,$ there exists a smooth conic metric $g\in\mathcal{C}(S,D)$
such that
\[
\rho(g)\geq\rho_{0}(S,D)-\epsilon.
\]

\end{thm}
We shall construct a conformal factor $u_{2}$ whose pinching constant
is close to $\rho_{0}(S,D)$. The construction is divided into three
steps. First, we construct an approximate conformal factor $u_{0}$
based on our earlier work \cite{FL}, which has desired singularities
along $D$, but it has discontinuity along $\{|z|=1\}.$ Second,
we run a mollification argument to get a conformal factor $u_{1}$,
which smoothes out the discontinuity of $u_{0}$. Finally, we combine
$u_{0}$ and $u_{1}$ using a standard cutoff function to get the
desired function $u_{2}.$
\begin{proof}
We now describe our approximate conformal factor $u_{0}$. For $n=1,2,$
$u_{0}$ can be chosen as the conformal factor of the glued football
(up to a constant) with the corresponding cone angles. Namely, by
(\ref{eq:italian football}), we define
\begin{equation}
u_{0}(z)=\begin{cases}
\ln(\frac{2|z|^{\alpha}}{1+|z|^{2+2\alpha}}), & |z|<1;\\
\\
\ln(\frac{2|z|^{\beta_{1}}}{1+|z|^{2+2\beta_{1}}}) & |z|\geq1.
\end{cases}
\end{equation}
Thus, we have that $u_{0}\in C^{1,1}(\mathbb{C})$ for $n=1$ and
$u_{0}\in C^{1,1}(\mathbb{C}\backslash\{0\})$ for $n=2.$

For the case $n\geq3,$ to construct $u_{0},$ we need to apply the
main result in \cite{FL}. Given a supercritical divisor $D$ on
$S=\mathbb{S}^{2}$, where $D=\sum_{i=1}^{n}\beta_{i}p_{i}$, assume $\beta_1=\min \beta_i$ and let
$\alpha=|D|-\beta_{1}$. We then have $\beta_{1}<\alpha<0.$ Define a monotone
decreasing sequence $\{\alpha_{j}\}_{j=1}^{\infty}$, such that $\alpha_{j}<\beta_{2},$
and
\[
\lim_{j\to\infty}\alpha_{j}=\alpha.
\]
We fix $p_{0}=\infty$ and consider divisors
\begin{equation}
D_{j}=\sum_{i>1}\beta_{i}p_{i}+\alpha_{j}p_{0}.
\end{equation}
$(S,D_{j})$ is subcritical for each $j$. Therefore, due to \cite{Tr},
there exists a conic metric $g_{j}=e^{2v_{j}}g_{0}$ for the pair
$(S,D_{j})$ such that $K(g_{j})=(\alpha_{j}+1)^{2}$. By the main
theorem of \cite{FL}, we know that after a suitable normalization
and passing to a subsequence if necessary,
\begin{equation}
v_{j}(z)\to v_{\infty}(z)=\ln(\frac{2|z|^{\alpha}}{1+|z|^{2+2\alpha}}),\quad j\to\infty,\label{eq:add10}
\end{equation}
where the convergence is $C^{\infty}$ on any compact $K\subset\mathbb{C}\backslash{O}$.
Note that the convergence (\ref{eq:add10}) implies that the singular points of $v_{j}$  converge
to the origin $O$ when $j\to\infty$. Without loss of generality,
we may assume that $v_{j}$ is smooth for $1/4<|z|<\infty.$ Define
\begin{equation}
w_{j}(z)=\begin{cases}
v_{j}(z), & |z|<1;\\
\\
\ln(\frac{2|z|^{\beta_{1}}}{1+|z|^{2+2\beta_{1}}}) & |z|\geq1.
\end{cases}
\end{equation}
Consider a general piecewise smooth function $w$ defined in the region
$\{|z|\geq1/2\}$ with the discontinuity at $\{|z|=1\}$. Define the
following
\[
D_{z}^{0}(w)=\limsup_{z_{1,}z_{2}\to z}|w(z_{1})-w(z_{2})|,
\]
\[
D_{z}^{1}(w)=\limsup_{z_{1},z_{2}\to z}|\nabla w(z_{1})-\nabla w(z_{2}))|.
\]
Following (\ref{eq:add10}), we get:
\begin{prop}
For any $\epsilon'>0$, there exists a $J\in\mathbb{N}$ such that
for $j\geq J,$ $z\in\{2>|z|>1/2\},$ $i=0,1,$ we have\end{prop}
\begin{enumerate}
\item
\[
|\nabla^{p}(v_{j}(z)-v_{\infty}(z))|\leq\epsilon'\;p=0,1,2,3;
\]

\item
\[
D_{z}^{i}(w_{j})\leq\epsilon'.
\]

\end{enumerate}
We will pick our approximate conformal factor $u_{0}$ as one of the
$w_{j}$'s, with the choice of $j$ given later.

To summarize properties of $u_{0}$, we have the following
\begin{prop}
\label{prop:u_0} For any $\epsilon'>0,$ there exists a function
$u_{0}:\mathbb{C\to\mathbb{R}},$ such that
\begin{enumerate}
\item $u_{0}$ is smooth away from the curve $\{|z|=1\},$ and $z_{2},\cdots,z_{n}\in\{|z|<1/2\}$;
\item when $u_{0}$ is smooth, we have\label{enu:initialbound}
\[
(\beta+1)^{2}\leq-e^{-2u_{0}}\Delta u_{0}\leq(\alpha+1)^{2};
\]

\item $u\sim\beta_{i}\ln|z-z_{i}|$ as $z\to z_{i}$, $i>1$;
\item $u\sim-(\beta_{1}+2)\ln|z|$ as $|z|\to z_{1}=\infty$;
\item $D_{z}^{i}(u_{0})\leq\epsilon'$, for $|z|>\frac{1}{2}$.
\item There are constants $m_{0}$, $M_{0}$ and $M_{1}$ depending only
on $\alpha$ and $\beta_{1}$ such that
\begin{align*}
M_{0} & \geq\sup_{|z|>1/2}u_{0}(z),\\
M_{1} & \geq\sup_{1/2\leq|z|,|z|\neq1}|\nabla u_{0}(z)|,\\
m & \leq\inf_{1/2\leq|z|\leq2}u_{0}(z).
\end{align*}

\end{enumerate}
\end{prop}
Next, we describe our mollification procedure. Define, for any $z\in\mathbb{C}$
and $\delta>0$,
\[
\varphi(z)=\varphi_{\delta}(z):=\begin{cases}
\frac{c}{\delta^{2}}\exp(\frac{-\delta^{2}}{(|z|-\delta)^{2}}), & 0\leq|z|<\delta,\\
0, & |z|\geq\delta,
\end{cases}
\]
where $c\in\mathbb{R}$ is chosen such that $\int_{\mathbb{C}}\varphi dv=1$.
There exists a constant $C>10$, such that
\begin{equation}
|\varphi|\leq\frac{C}{\delta^{2}},\quad|\nabla\varphi|\leq\frac{C}{\delta^{3}}.\label{eq:cutoff}
\end{equation}
 Define, for $|z|\geq1/2$,
\begin{equation}
u_{1}(z)={\displaystyle \int_{w\in\mathbb{C}}\varphi(z-w)u_{0}(w)}dv_{w}.\label{eq:u_1}
\end{equation}
We prove the following
\begin{prop}
\label{prop:u_1}For any $\epsilon>0$, and there exists $\epsilon'>0$
and $1/8>\delta>0$, such that if $u_{0}$ is a function satisfying
Proposition \ref{prop:u_0}, then the function $u_{1}$ defined in
(\ref{eq:u_1}) satisfies the following
\begin{enumerate}
\item $u_{1}$ is smooth in the region $\{z:\:|z|\geq5/8\}$;
\item In the region $\{1/4\leq||z|-1|\leq3/8\}$, \label{C_2}
\[
||u_{1}-u_{0}||_{C^{2}}\leq\epsilon;
\]

\item \label{enu:key-1}For $3/4\leq|z|\leq5/4,$
\[
(\beta+1)^{2}-\epsilon\leq-e^{-2u_{1}}\Delta u_{1}\leq(\alpha+1)^{2}+\epsilon.
\]

\end{enumerate}
\end{prop}
\begin{proof}
The proof of parts (1) and (2) of Proposition \ref{prop:u_1} is standard.
We just need to prove part (\ref{enu:initialbound}).

Notice that
\[
\nabla_{z}\varphi(z-w)=-\nabla_{w}\varphi(z-w),
\]
 we have the following inequalities for $|z|\geq5/8$,
\begin{equation}
|u_{1}(z)-u_{0}(z)|=|\int\varphi(z-w)(u_{0}(z)-u_{0}(w))dv_{w}|\leq M_{1}\delta,
\end{equation}
\begin{align}
|\Delta u_{1}(z)-\int\varphi(z-w)\Delta u_{0}(w)dv_{w}|\nonumber \\
\leq\intop_{\{|w|=1,|w-z|\leq\delta\}}|\nabla\varphi(z-w)|D_{w}^{0}(u_{0})+\varphi D_{w}^{1}(u_{0}) & \leq2\pi\delta(\frac{C\epsilon'}{\delta^{3}}+\frac{C\epsilon'}{\delta^{2}}).\nonumber \\
\label{eq:add9}
\end{align}
 Now we can choose $\epsilon'=\delta^{3}$ in (\ref{eq:add9}) to
get
\begin{equation}
|\Delta u_{1}(z)-\int\varphi(z-w)\Delta u_{0}(w)dv_{w}|\leq4\pi C\delta.
\end{equation}
Notice that if $|z-w|\leq\delta<1/8$, $|z|\geq3/4,$ we have $|w|\geq5/8$,
and
\[
|u_{0}(z)-u_{1}(w)|\leq|u_{0}(w)-u_{1}(w)|+|u_{0}(w)-u_{0}(z)|\leq2M_{1}\delta.
\]
For $3/4\leq|z|\leq5/4$, $|z|\neq1$,
\begin{align}
-e^{-2u_{1}(z)}\Delta u_{1}(z) & \leq e^{-2u_{1}(z)}[(\int\varphi(z-w)(\alpha+1)^{2}e^{2u_{0}(w)})+4\pi C\delta]\nonumber \\
 & \leq4\pi Ce^{-2m}\delta+(\alpha+1)^{2}e^{4M_{0}\delta}.
\end{align}
Similarly, we can prove that
\[
-e^{-2u_{1}(z)}\Delta u_{1}(z)\geq(\beta_{1}+1)^{2}e^{-4M_{0}\delta}-4\pi Ce^{-2m}\delta.
\]
Thus for any given $\epsilon>0,$ we can choose a proper $\delta<<1$
to get (\ref{enu:key-1}) for $3/4\leq|z|\leq5/4$, $|z|\neq1$. Now
that $u_{1}$ is smooth for $3/4\leq|z|\leq5/4$, we can thus extend
this estimate to get (3).
\end{proof}
Finally, we describe our smooth conic metric whose curvature pinching
is arbitrarily close to $\rho_{0}.$ Define a cut-off function $\chi\in C^{\infty}(\mathbb{C})$
such that
\begin{enumerate}
\item
\[
\chi(z)=\begin{cases}
1, & z\in R_{1}=\{3/4\leq|z|\leq5/4\};\\
\\
0, & z\in R_{2}=\{|z|<5/8,\mathrm{\,or}\,|z|>11/8\};
\end{cases}
\]

\item
\begin{equation}
0\leq\chi(z)\leq1,\,|\nabla\chi|<16,\,|\Delta\chi|\leq256,\;|z|\in\mathbb{C}\backslash(R_{1}\cup R_{2}).\label{eq:chi}
\end{equation}

\end{enumerate}
We define the following function
\begin{equation}
u_{2}(z)=\chi u_{1}+(1-\chi)u_{0}.
\end{equation}
We prove the following
\begin{prop}
\label{prop:u_2}Function $u_{2}$ satisfies the following:
\begin{enumerate}
\item $u_{2}$ is smooth away from $z_{2},\cdots,z_{n}\in\{|z|<1/4\}$;
\item $u_{2}\sim\beta_{i}\ln|z-z_{i}|$ as $z\to z_{i}$, for $i=2,\cdots,n$;
\item $u_{2}\sim-(\beta_{1}+2)\ln|z|$ as $|z|\to z_{1}=\infty$;
\item In the region where $u_{2}$ is smooth, there exists a constant $C"=C"(\alpha,\beta_{1})>0$
such that
\begin{equation}
(\beta_{1}+1)^{2}-C"\epsilon\leq-e^{-2u_{2}}\Delta u_{2}\leq(\alpha+1)^{2}+C"\epsilon;\label{eq:last}
\end{equation}

\end{enumerate}
\end{prop}
\begin{proof}
Since $u_{2}(z)=u_{0}(z)$ for $z\in R_{2},$ and $u_{2}=u_{1}$ for
$z\in R_{1},$ parts (1), (2) and (3) are obvious. Apply (\ref{enu:key-1})
of Proposition \ref{prop:u_1}, we can prove (4) for $z\in R_{1}\cup R_{2}$.

Thus, to prove Proposition \ref{prop:u_2}, we just need to verify
part (4) for $z\in R=\mathbb{C}\backslash(R_{1}\cup R_{2})=\{5/8\leq|z|\leq3/4,\:\mathrm{{or}\:}5/4\leq|z|\leq11/8\}$.
Notice that $u_{2}-u_{0}=\chi(u_{1}-u_{0})$, thus by (\ref{C_2})
and (\ref{eq:chi}), we have a finite constant $C'>0$, for any $z\in R$
such that $|z|\neq1$,
\[
|\Delta(u_{2}-u_{0})|=|\chi\Delta(u_{1}-u_{0})+2\nabla\chi\nabla(u_{1}-u_{0})+(\Delta\chi)(u_{1}-u_{0})|\leq C'\epsilon,
\]
which leads to
\[
|e^{-2u_{2}}\Delta u_{2}-e^{-2u_{0}}\Delta u_{0}|\leq e^{-2u_{2}}|\Delta(u_{2}-u_{0})|+|(e^{-2u_{0}}-e^{-2u_{2}})\Delta u_{0}|
\]
\[
\leq e^{-2m+2\epsilon}C'\epsilon+|e^{-2u_{0}}\Delta u_{0}(1-e^{-2(u_{2}-u_{0})})|\leq C"\epsilon.
\]
The last inequality follows from the fact that $-e^{-2u_{0}}\Delta u_{0}(z)=(\beta_{1}+1)^{2}$
or $(\alpha+1)^{2}$ for $z\in R$. Thus we have proved part (4) in
$R$.

Combine Propositions \ref{prop:u_0}, \ref{prop:u_1} and \ref{prop:u_2},
we have effectively proved Theorem \ref{thm:construction}.
\end{proof}
Theorem \ref{mainthm} thus easily follows Proposition \ref{prop:keyestimate}
and Theorem \ref{thm:construction}.
\end{proof}

\section{Convergence}

In this section, we prove Theorem \ref{thm:converging}.
Tracing the proof of Proposition \ref{prop:keyestimate}, it is easy
to check that if the best constant $\rho_{0}$ is achieved by a conformal
factor $u$, level sets of $u$ are concentric round circles; thus,
$u$ has to be radially symmetric. A bit further computation shows
that $u$ has to be the conformal factor of a glued football, which
is $C^{1,1}$ and piece-wise smooth away from one or two singular
points. This fact was pointed out by Chen-Lin \cite{ChLi} and Bartolucci
\cite{B} for the single and double singular points cases respectively
using the symmetric rearrangement argument. In multiple conic points cases, equality case of Proposition
\ref{prop:keyestimate} cannot be expected. We will do a finer analysis on
the isoperimetric defect to show that all but one of singular points merge to one conic
point when the best pinching constant $\rho_{0}$ is approximated. In
\cite{FL}, we have described exactly this kind of merging behavior
for conic 2-spheres with constant curvature metrics. We thus follow
the arguments given in \cite{FL}, pointing out only the necessary
modification for the supercritical case.

First, we prove a technical lemma.
\begin{lem}
\label{lem:technical}Let $u\in C^{\infty}(\Omega)$ be a solution
of the Dirichlet problem in a bounded region $\Omega\subset\mathbb{C}$,
\[
\begin{cases}
\Delta u=-Ke^{2u}, & \text{in \ensuremath{\Omega};}\\
u=s, & \text{on \ensuremath{\partial\Omega},}
\end{cases}
\]
where $K$ is a positive continuous function with $0<a\leq K\leq b$. Let
$\Omega_{t}:=\{u>t\}\subset\Omega$, $A(t)=\intop_{\Omega_{t}}Ke^{2u}$,
$B(t)=|\Omega_{t}|,$ and $H=\max_{z\in\Omega}u(z),$ then
\[
A(t)\geq\frac{4a\pi}{b}(1-e^{t-H}),
\]
moreover for $A(t)\geq\frac{2a\pi}{b},$ we have, for any $t\geq s,$
\[
B(t)\geq\frac{4a\pi}{b^{2}}(e^{-t-H}-e^{-2H}).
\]
\end{lem}

\begin{proof}
We are in a similar but simpler set up as that of Proposition \ref{prop:keyestimate},
as no singularity appears here. In particular, similar to (\ref{eq:1}),
we have
\begin{equation}
be^{2t}B'\leq A'\leq ae^{2t}B'.\label{eq:add6}
\end{equation}
Thus we follow the proof of Lemma 3.4 of \cite{FL} to get
\begin{equation}
(A^{2})'=2(\intop_{\partial\Omega(t)}|\nabla u|)(-e^{2t}\intop_{\partial\Omega(t)}\frac{K}{|\nabla u|})\leq-2e^{2t}a|\partial\Omega(t)|^{2}\leq-8a\pi e^{2t}B.\label{eq:add8}
\end{equation}
We integrate (\ref{eq:add8}) from $t$ to $H$ to get
\begin{equation}
A^{2}(t)\geq8a\pi\intop_{t}^{H}e^{2\mu}B(\mu)d\mu.\label{eq:add7}
\end{equation}
On the other hand, integrating (\ref{eq:add6}), we get
\begin{align}
-A(t) & \geq-\intop_{t}^{H}2be^{2\mu}B(\mu)d\mu-be^{2t}B(t).\label{eq:add5}
\end{align}
Thus we combine (\ref{eq:add7}) and (\ref{eq:add5}) to get
\begin{align}
\frac{b}{4a\pi}A^{2} & \geq A-be^{2t}B\label{eq:add13}\\
 & \geq\frac{b}{4a\pi}AA'+A,
\end{align}
 which leads to
\[
A(t)\geq\frac{4a\pi}{b}(1-e^{t-H}).
\]
When $A\geq\frac{2a\pi}{b},$ by (\ref{eq:add13}),
\[
B(t)\geq\frac{1}{b}e^{-2t}A(1-\frac{b}{4a\pi}A)\geq\frac{4a\pi}{b^{2}}[e^{-t-H}-e^{-2H}].
\]
\end{proof}
\begin{rem}
As in \cite{FL}, Lemma \ref{lem:technical} is used to show the uniform
upper bound for conformal factors in consideration. A more general
form of such estimates has been obtained by Brezis-Merle \cite{BM}.
\end{rem}
We can now start the proof of Theorem \ref{thm:converging}.
\begin{proof}
We write
\[
g_{i}=e^{2u_{i}}g_{0},
\]
which has conic singularity along the divisor $D.$ We will normalize
$u_{i}$ later. Following notations of our earlier works \cite{FL},
we define
\begin{align*}
\Omega_{i}(t): & =\{u_{i}>t\}\subset\mathbb{C},\ \ \ \ \ A_{i}(t):=\int_{\Omega_{t}}Ke^{2u_{i}},\ \ \ \ \ B_{i}(t):=|\Omega_{i}(t)|,
\end{align*}
where all integrals are with respect to the Euclidean metric $g_{0}$.
Note that, under our set-up, the Gauss-Bonnet formula can be written
as
\begin{equation}
\int_{\mathbb{R}^{2}}K(g_{i})e^{2u}=2\pi(2+|D|)=\lim_{t\to-\infty}A_{i}(t);\label{eq:limitGB-1}
\end{equation}
while we also have that
\[
A_{i}(t)=\int_{\Omega_{i}(t)}Ke^{2u_{i}}=\int_{\partial\Omega_{i}(t)}|\nabla u_{i}|+2\pi\alpha,
\]
 for any $t>-\infty$.

According to the proof in Proposition~\ref{prop:keyestimate}, we have
\begin{equation}
B_{i}'(t)=-\int_{\partial\Omega_{i}(t)}\frac{1}{|\nabla u_{i}|},\quad A'(t)=-e^{2t}\int_{\partial\Omega_{i}(t)}\frac{K(g_{i})}{|\nabla u_{i}|},\label{eq:add14}
\end{equation}
for $t$ almost everywhere.

Assuming that
\begin{equation}
a_{i}=\min K(g_{i})\leq K(g_{i})\leq b_{i}=\max K(g_{i}). \label{eq:add15}
\end{equation}
By adding a proper constant to each $u_{i}$, we may assume $b_{i}=1$,
which leads to
\begin{equation}
a_{i}\to\rho_{0}=\frac{(1+\beta_{1})^{2}}{(1+\alpha)^{2}},\label{eq:add16}
\end{equation}
 as $i\to\infty.$ Thus, combining (\ref{eq:add14}) and (\ref{eq:add15}),
we get
\begin{equation}
1\leq\frac{e^{2t}B_{i}'(t)}{A_{i}'(t)}\leq\frac{(1+\alpha)^{2}}{(1+\beta_{1})^{2}}.\label{eq:1-1}
\end{equation}
Also by the isoperimetric inequality and the Hölder's inequality,
we have
\begin{align}
4\pi B_{i}(t)\leq(\int_{\partial\Omega_{i}(t)}1)^{2}\leq\int_{\partial\Omega_{i}(t)}|\nabla u_{i}|\int_{\partial\Omega_{i}(t)}\frac{1}{|\nabla u_{i}|}=-B_{i}'(t)(A_{i}(t)-2\pi\alpha).\label{2-1}
\end{align}
Similar to the discussion in \cite{FL}, due to the non compact conformal
transformation group, there are two families of normalization that
we can apply to functions $\{u_{i}\}$ without changing the geometric
setting. Namely,
\[
\text{scaling: \ensuremath{u^{\lambda,0}(z):=u(\lambda z)+\ln\lambda;}}
\]
\[
\text{translation: \ensuremath{u^{0,k}(z):=u(z-k).}}
\]
We choose the normalization so that for a generic $t_{0}\in\mathbb{R},$
\begin{align*}
A_{i}(\ln(\alpha+1)) & =2\pi(\alpha+1),\\
\text{\normalcolor the\,centroid\,of} & \text{ \ensuremath{\Omega_{i}(t_{0})} is at 0.}
\end{align*}

Define
\begin{equation}
D_{i}(t)=(\int_{\partial\Omega_{i}(t)}1)^{2}-4\pi B_{i}(t)
\end{equation}
as the isoperimetric defect for the region $\Omega_{i}(t).$ We have
the following improvement of (\ref{eq:2.5}):
\[
4\pi B_{i}(t)+D_{i}(t)\leq-B_{i}'(t)(A_{i}(t)-2\pi\alpha),
\]
which means
\begin{equation}
-\frac{B_{i}(t)}{B_{i}'(t)}\leq\frac{A_{i}(t)-2\pi\alpha}{4\pi}+\frac{D_{i}(t)}{4\pi B_{i}^{'}(t)}.\label{eq:2.5-1}
\end{equation}
Thus, similar to (\ref{eq:add30})
\begin{align}
\frac{d}{dt}[e^{2t}B_{i}(t)] & =e^{2t}B_{i}'(t)+2e^{2t}B_{i}(t)=\frac{e^{2t}B_{i}'(t)}{A_{i}'(t)}A_{i}'(t)(1+\frac{2B_{i}(t)}{B_{i}'(t)})\nonumber \\
\notag & \leq\frac{e^{2t}B_{i}'(t)}{A_{i}'(t)}A_{i}'(t)(1-\frac{A_{i}-2\pi\alpha}{2\pi})-\frac{e^{2t}D_{i}(t)}{2\pi}\nonumber \\
 & \leq\begin{cases}
\frac{A_{i}'}{b_{i}}(1+\alpha-\frac{A_{i}}{2\pi})-\frac{e^{2t}D_{i}(t)}{2\pi}, & \;A_{i}\leq2\pi(1+\alpha);\\
\\
\frac{A_{i}'}{a_{i}}(1+\alpha-\frac{A_{i}}{2\pi})-\frac{e^{2t}D_{i}(t)}{2\pi}, & \;A_{i}>2\pi(1+\alpha).
\end{cases}\label{eq:add20}
\end{align}
Integrating (\ref{eq:add20}) from some $t\leq\ln(\alpha+1)$ to $\infty$,
we get
\[
\frac{1}{2\pi}\int_{t}^{\infty}e^{2s}D_{i}(s)ds-e^{2t}B(t)\leq\frac{1}{4a_{i}\pi}A^{2}(t)-\frac{1+\alpha}{a_{i}}A(t)+\pi(1+\alpha)^{2}(\frac{1}{a_{i}}-1).
\]
Takeing a sequence $t_n\to-\infty$ with $e^{2t_n}B(t_n)\to 0$, and noticing (\ref{eq:limitGB-1})
and (\ref{eq:add16}), we then get
\[
\int_{-\infty}^{\infty}e^{2t}D_{i}(t)dt\to0,\ \ \ \ \ \ \ \text{as \ensuremath{i\to\infty.}}
\]
 Thus, away from a set $\mathbb{\mathcal{S}}$ of measure $0$, we have
\[
D_{i}(t)\to0,\qquad i\to\infty.
\]
In particular, we pick our $t_{0}\notin\mathcal{S}$. A similar argument
as given in Lemma 3.5 of \cite{FL} indicates that
\[
A_{i}(t)\to A(t),\ \ \ \ B_{i}(t)\to B(t),
\]
 and all inequalities in the discussion above will be equalities when
passing to the limit. Therefore, by (\ref{eq:2.5-1}) and (\ref{eq:add20}), $A(t)$
and $B(t)$ satisfy the following

\begin{align}
4\pi B(t)= & -B'(t)(A(t)-2\pi\alpha),\label{eq:limit}\\
e^{2t}B(t) & =\begin{cases}
-\frac{A^{2}(t)}{4\rho_{0}^{2}\pi}+\frac{1+\alpha}{\rho_{0}}A(t), & t\geq\ln(\alpha+1);\\
\frac{-1}{4\rho_{0}\pi}A^{2}(t)+\frac{1+\alpha}{\rho_{0}}A(t)-\pi(1+\alpha)^{2}(\frac{1}{\rho_{0}}-1), & t<\ln(\alpha+1),
\end{cases}\\
K(g_{u}) & =\begin{cases}
1, & \qquad t\geq\ln(\alpha+1);\\
\rho_{0,} & \qquad t<\ln(\alpha+1).
\end{cases}
\end{align}
Combining these with proper initial conditions, we can compute $A(t)$
and $B(t)$ precisely. Readers are referred to \cite{FL} for explicit
formulae. It is straight forward to see that they are given by the
corresponding data of the glued football.

We now follow \cite{FL} to prove the Gromov-Hausdorff convergence.
Let $M_{i}(t)$ be the connected component of $\Omega_{i}(t)$with
the largest area. Since $D_{i}(t)\to0$, we apply Benneson's inequality
to get
\[
|M_{i}(t)|\to B(t).
\]
Due to the normalization, the centroid of $\Omega_{i}(t_{0})$ is
0, we conclude that $M_{i}(t_{0})$ converge to a round disc in Gromov-Hausdorff
sense. Indeed, following Lemma 3.9 of \cite{FL}, $M_{i}(t)$ converges
in Hausdorff distance to a disk $D(t)$ for almost every $t$. Let
$p_{0}$ be the limit of center of $D(t)$ as $t\to\infty$. Passing
to a subsequence if necessary, let $p_{1},\cdots,p_{n-1}$ be the
possible limit points of $n-1$ conic points. Consider any compact
set $K\subset\mathbb{C}\backslash\{p_{0},p_{1},\cdots,p_{n-1}\}$.
Following exactly the argument given in \cite{FL}, with Lemma \ref{lem:technical},
we can show that there exist constants $N\in\mathbb{N}$ and a uniform
constant $C_{K}\in\mathbb{R}$, such that $||u_{i}||_{C^{0}(K)}\leq C_{K}$
for $i\geq N$. Thus, by a standard bootstrap argument we have, up
to a subsequence,
\[
u_{i}\longrightarrow u,\text{ }
\]
 in $C^{\infty}(K)$ topology. It follows that the limit $u$ must
be radially symmetric, and its associated $A(t)$ and $B(t)$ are
given by (\ref{eq:limit}). We also have that $p_{0}=O$. A straightforward
computation shows that $u=u_{\alpha,\beta_{1}}$, the conformal factor
of the glued football $S_{\alpha,\beta_{1}}^{2}$.

It is now straightforward to see that $p_{i}=O$ for $i=1,\cdots,n-1$,
which means $n-1$ of the conic points collapse into one.
\end{proof}
The proof of Theorem \ref{thm:critical} follows exactly the same
line of those of Theorems \ref{mainthm} and \ref{thm:converging},
thus we omit it here.

\end{document}